\documentclass[9pt,reqno,a4]{amsart}

\numberwithin{equation}{section}

\usepackage[latin1]{inputenc}
\usepackage[english]{babel}

\usepackage{amsmath,amsthm,amsfonts,latexsym,amssymb}
\usepackage[colorlinks]{hyperref}
\hypersetup{linkcolor=blue,citecolor=blue,filecolor=black,urlcolor=blue}
\usepackage{comment}
\usepackage{soul}

\usepackage[
  hmarginratio={1:1},     % equal left and right margins
  vmarginratio={1:1},     % equal top and bottom margins
  textwidth=15cm,        % new text width  
  textheight=22cm,  
  heightrounded,          % always useful
]{geometry}

{ \theoremstyle{plain}
\newtheorem{theorem}{Theorem}[section]
\newtheorem{proposition}[theorem]{Proposition}
\newtheorem{lemma}[theorem]{Lemma}
\newtheorem{corollary}[theorem]{Corollary}
\newtheorem{claim}{Claim}
  \theoremstyle{remark}

  \theoremstyle{definition}

}

\newcommand\R{\text{I\!R}}
\newcommand\N{\text{I\!N}}

\newcommand\ve{\varepsilon}
\newcommand\e{\epsilon}
\newcommand\de{\delta}
\newcommand\be{\beta}
\newcommand\al{\alpha}

\newcommand{\Om}{\Omega}

\newcommand{\fr}{\partial}

\newcommand{\grad}{\nabla}
\newcommand{\ml}{\mathcal}

\newcommand{\sm}{\setminus}
\newcommand{\la}{\lambda}
\newcommand{\dem}{\bf Proof:}

\newcommand\lap{\Delta}

\newcommand\ti{\tilde}

\newcommand{\lf}{\left}
\newcommand{\rg}{\right}

\newcommand\ds{\displaystyle}

\renewcommand\({\left(}
\renewcommand\){\right)}

\DeclareMathAlphabet{\mathpzc}{OT1}{pcz}{m}{it}

\begin{document}   \title[sinh-Poisson type equation on pierced domains]{A note on a sinh-Poisson type equation with variable\\ intensities on pierced domains}

\author[P. Figueroa]{Pablo Figueroa}
\address{Pablo Figueroa 
\newline \indent Universidad Austral de Chile
\newline \indent Facultad Ciencias
\newline \indent Instituto de Ciencias Físicas y Matemáticas
\newline \indent Campus Isla Teja, Valdivia, Chile }
\email{pablofs78@gmail.com }

\date{\today}


\maketitle

\begin{abstract}
\noindent 
We consider a sinh-Poisson type equation with variable intensities and Dirichlet boundary condition on a pierced domain
\begin{equation*}
\left\{ \begin{array}{ll} 
\Delta u +\rho\lf(V_1(x)e^{u}- V_2(x)e^{-\tau u}\rg)=0 &\text{in } \Om_\e:=\Om\sm \ds \bigcup_{i=1}^m \overline{B(\xi_i,\e_i)}\\
u=0&\text{on }\fr\Om_\e,
\end{array}\right.
\end{equation*}
where $\rho>0$, $V_1,V_2>0$ are smooth potentials in $\Om$, $\tau>0$, $\Om$ is a smooth bounded domain in $\R^2$ and  $B(\xi_i,\e_i)$ is a ball centered at $\xi_i\in \Om$ with radius $\e_i>0$, $i=1,\dots,m$. When $\rho>0$ is small enough and $m_1\in \{1,\dots,m-1\}$, there exist radii $\e=(\epsilon_1,\dots,\epsilon_m)$ small enough such that the problem has a solution  which blows-up positively at the points $\xi_1,\dots,\xi_{m_1}$ and negatively  at the points $\xi_{m_1+1},\dots,\xi_{m}$ as $\rho\to 0$. The result remains true in cases $m_1=0$ with $V_1\equiv 0$ and $m_1=m$ with $V_2\equiv 0$, which are Liouville type equations.
\end{abstract}

\tableofcontents

\section{Introduction}

Let  $\Om\subset \R^2$ be a smooth bounded domain. Given $\e:=(\epsilon_1,\dots,\epsilon_m)$ and $m$ different points $\xi_1,\dots,\xi_m\in\Om$, define $\Om_\e:=\Om \setminus \cup_{i=1}^m \overline{B(\xi_i,\e_i)},$ a pierced domain, where
  $B(\xi_i,\e_i)$ is  a ball centered at $\xi_i$ with radius $\e_i >0$. Inspired by recent results in \cite{EFP}, we are interested in this paper in building solutions to a sinh-Poisson type equation with variable intensities and Dirichlet boundary condition on pierced domains:
\begin{equation}\label{lepd}
\left\{ \begin{array}{ll} 
\Delta u +\rho\lf(V_1(x)e^{u}-\nu V_2(x) e^{-\tau u}\rg)=0 &\text{in }\Om_\e\\
u=0&\text{on }\fr\Om_\e,
\end{array}\right.
\end{equation}
where $\rho>0$ is small, $V_1,V_2>0$ are smooth potentials in $\Om$, $\nu\ge0$ and $\tau>0$. This equation and its variants have attracted a lot of attention in recent years due to its relevance in the statistical mechanics description of 2D-turbulence, as initiated by Onsager \cite{o}. Precisely, in this context Caglioti, Lions, Marchioro, Pulvirenti \cite{clmp} and Sawada, Suzuki \cite{ss} derive the following equation:
\begin{equation}\label{p1}
\left\{\begin{array}{ll}
\ds -\Delta u=\lambda \int\limits_{[-1,1]}{\alpha e^{\alpha u}\over \int\limits_\Omega e^{\alpha u}dx}d \mathcal P(\alpha)& \hbox{in}\ \Omega \\
 u=0 & \hbox{on}\ \partial\Omega,
\end{array}\right.
\end{equation}
where $\Omega$ is a bounded domain in $\mathbb R^2,$ $u$ is the stream function of the flow, $\lambda>0$ is a constant related to the inverse temperature and  $\mathcal P$ is a Borel probability measure in $[-1,1]$ describing the point-vortex intensities distribution. We observe that \eqref{p1} is obtained under a \emph{deterministic} assumption on the distribution of the vortex circulations. 

%\medskip \noindent On the other hand, a similar mean field equation to \eqref{p1} is derived by Neri in \cite{N}, under the \emph{stochastic} assumption that the point vortex circulations are independent identically distributed random variables with probability distribution $\mathcal P$. For Neri's model, several blow-up and existence results have been obtained by Ricciardi and Zecca in \cite{rz2016}. Moreover, the same authors show some common properties between such deterministic and stochastic models in \cite{rz2012}.

\medskip \noindent On the other hand, on a bounded domain $\Om\subset\mathbb R^2$ a similar mean field equation to \eqref{p1} is derived by Neri in \cite{N}:
\begin{equation}\label{nerieq}
\left\{\begin{array}{ll}
\ds -\Delta u=\lambda \int\limits_{[-1,1]}{\alpha e^{\alpha u}d\mathcal P(\al) \over \int\!\int_{\Omega\times[-1,1]} e^{\alpha u} d\mathcal P(\alpha)dx}& \hbox{in}\ \Omega \\
 u=0 & \hbox{on}\ \partial\Omega,
\end{array}\right.
\end{equation}
under the \emph{stochastic} assumption that the point vortex circulations are independent identically distributed random variables with probability distribution $\mathcal P$. For Neri's model, several blow-up and existence results have been obtained by Ricciardi and Zecca in \cite{rz2016}. Moreover, the same authors show some common properties between such deterministic and stochastic models in \cite{rz2012}.

\medskip \noindent Equation \eqref{p1} (and also \eqref{nerieq}) includes several well-known problems depending on a suitable choice of $\ml P$. For instance, if $\mathcal P=\delta_1$ is concentrated at $1$, then \eqref{p1} corresponds to the classical mean field equation
\begin{equation}\label{p2}
\left\{\begin{array}{ll} 
\ds -\Delta u=\lambda  { e^{ u}\over \int\limits_\Omega e^{  u}dx}& \hbox{in}\ \Omega\\
 u=0& \hbox{on}\ \partial\Omega,
 \end{array}\right.
 \end{equation}
which has been widely studied in the last decades as shown in \cite{l}.  %Note that equation \eqref{nerieq} also includes the classical mean field equation \eqref{p2} assuming $\mathcal P=\de_1$. 
When $\mathcal P=\sigma \delta_{1}+(1-\sigma)\delta_{-\tau }$ with $\tau \in[-1,1]$ and $\sigma\in[0,1]$, equation \eqref{p1} becomes
\begin{equation}\label{p21}
\left\{\begin{array}{ll} 
\ds -\Delta u=\lambda \left( \sigma{ e^{ u}\over \int\limits_\Omega e^{  u}dx} -(1-\sigma)\tau { e^{ -\tau  u}\over \int\limits_\Omega e^{ -\tau   u}dx}
\right)& \hbox{in}\ \Omega\\
 u=0& \hbox{on}\ \partial\Omega.
 \end{array}\right.
 \end{equation}
Notice that solutions of \eqref{p2} are critical points of the functional
$$J_\la(u)={1\over 2}\int_\Om |\nabla u|^2-\la \log\lf(\int_\Om e^{u}\rg),\quad u\in H_0^1(\Om),$$
which can be found as minimizers of $J_\la$ if $\la<8\pi$, by using Moser-Trudinger's inequality. In the supercritical regime $\la\ge 8\pi$, the situation becomes subtler since the existence of solutions could depend on the topology and the geometry of the domain. In \cite{CL1,CL2}, Chen and Lin proved that \eqref{p2} has a solution when $\la\notin 8\pi\N$ and $\Om$ is not simply connected using a degree argument. On Riemann surfaces the degree argument in \cite{CL1,CL2} is still available and has received a variational counterpart in \cite{Dja,Mal} by means of  improved forms of the Moser-Trudinger inequality. When $\la=8\pi$ problem \eqref{p2} is solvable on a long and thin rectangle, as showed by Caglioti et al. \cite{clmp1}, but not on a ball. Bartolucci and Lin \cite{BL} proved that \eqref{p2} has a solution for $\la = 8\pi $ when the Robin function of $\Om$ has more than one maximum point.

\medskip \noindent 
Setting $\lambda_1=\lambda\sigma$, $\lambda_2=\lambda(1-\sigma)$ and $V_1=V_2=1$ problem \eqref{p21} can be rewritten as
\begin{equation}\label{p22}
\left\{\begin{array}{ll} 
\ds -\Delta u=\lambda_1  { V_1 e^{ u}\over \int\limits_\Omega V_1 e^{  u}dx} -\lambda_2\tau { V_2 e^{ -\tau  u}\over \int\limits_\Omega V_2 e^{- \tau   u}dx}& \hbox{in}\ \Omega\\
  u=0& \hbox{on}\ \partial\Omega.
  \end{array}\right.
  \end{equation}
If $\tau =1$ and $V_1=V_2\equiv 1$ problem \eqref{p22} reduces to mean field equation of the equilibrium turbulence, see \cite{bjmr,j,jwy2,os1,r} or its related sinh-Poisson version, see \cite{BaPi,BaPiWe,GP,jwy1,jwyz}, which have received a considerable interest in recent years.

\medskip \noindent To the extent of our knowledge, there are by now just few results in a  more general situation. Pistoia and Ricciardi built in \cite{pr1} sequences of blowing-up solutions   to \eqref{p22} when  $\tau  >0$ and $\lambda_1,\lambda_2\tau^2$ are close to $8\pi$, while in \cite{pr2} the same authors built an arbitrary large number of  sign-changing blowing-up solutions to \eqref{p22} when $\tau  >0$ and $\lambda_1,\lambda_2\tau^2$ are close to suitable (not necessarily integer) multiples of $8\pi.$  Ricciardi and Takahashi in \cite{rt}  provided a complete blow-up picture for solution sequences of \eqref{p22}   and successively  in \cite{rtzz} Ricciardi et al. constructed min-max solutions  when $\lambda_1 \to 8\pi^+$ and   $\lambda_2 \to 0$ on a multiply connected domain (in this case the nonlinearity $e^{-\tau  u}$ 
  may  be treated as a lower-order term with respect to the main term $e^u$). A blow-up analysis and some existence results are obtained when $\tau>0$ in a compact Riemann surface in \cite{j2,rz}.

\medskip
\medskip \noindent A matter of interest to us is whether do there exist solutions to \eqref{lepd} for small values of $\rho$ or \eqref{p22} for general values of the parameters $\lambda_1,\lambda_2>0$ on multiply connected domain $\Omega$. Ould-Ahmedou and Pistoia in \cite{op} proved that on a pierced domain $\Omega_\epsilon:=\Omega\setminus \overline{B(\xi_0,\epsilon)}$, $\xi_0\in\Omega$, there exists a solution to the classical mean field equation \eqref{p2} which blows-up at $\xi_0$ as $\epsilon \to 0$ for any $\lambda>8\pi$ (extra symmetric conditions are required when $\la \in 8\pi \mathbb{N}$). Recently, in \cite{EFP} the authors studied the mean field equation with variable intensities on pierced domains
\begin{equation}\label{mfevi}
\left\{ \begin{array}{ll} 
-\lap u=\la_1\dfrac{V_1 e^{u}}{ \int_{\Om_{\boldsymbol\epsilon}} V_1  e^{u} dx } - \la_2\tau \dfrac{ V_2 e^{-\tau u}}{ \int_{\Om_{\boldsymbol\epsilon}}V_2 e^{ - \tau u} dx}&\text{in $\Om_{\boldsymbol\epsilon}$}\\
\ \ u=0 &\text{on $\fr \Om_{\boldsymbol\epsilon}$},
\end{array} \right.
\end{equation}
in the super-critical regime $\lambda_1>8\pi m_1$ and $\lambda_2 \tau^2>8\pi (m-m_1)$ with $m_1 \in \{0,1,\dots,m\}$. This equation is related, but not equivalent, to problem \eqref{lepd} by using the change
$$\rho={\la_1\over \int_{\Om_\e} V_1 e^{u} } \qquad \text{ and }\qquad \rho\nu ={\la_2\tau \over \int_{\Om_\e} V_2 e^{-\tau u} }.$$
More precisely, the authors constructed solutions to \eqref{mfevi} $u\e$ in $\Om_\e$ blowing-up positively and negatively at $\xi_1,\dots,\xi_{m_1}$ and $\xi_{m_1+1},\dots,\xi_m$, respectively, as $\epsilon_1,\dots,\epsilon_m \to 0$ under the assumption
$$\la_1=4\pi(\al_1+\dots+\al_{m_1}), \quad \la_2 \tau^2 = 4\pi(\al_{m_1+1}+\dots+\al_{m }),\
m_1 \in \{0,1,\dots,m\},\ \al_i> 2,\ \al_i\not\in 2\mathbb N.$$ 
Nevertheless, the result in \cite{EFP} may not tell us whether \eqref{lepd} has solutions with $m_1$ positive bubbles and $m-m_1$ negative bubbles for \emph{all} small $\rho>0$. Therefore, we perform directly to problem \eqref{lepd} a similar procedure.
Our main result reads as follows.
\begin{theorem}\label{main}
Let $m$ be a positive integer and $m_1\in \{0,\dots, m\}$. Then, for all $\rho>0$ small enough there are radii $\e(\rho)=\lf(\epsilon_1(\rho),\dots,\epsilon_m(\rho)\rg)$ small enough such that the problem \eqref{lepd} has a solution $u_\rho$ in $\Om_\e$ blowing-up positively at $\xi_1,\dots,\xi_{m_1}$ and negatively at $\xi_{m_1+1},\dots,\xi_m$ as $\rho$ goes to zero.
\end{theorem}

\medskip \noindent In Theorem \ref{main} we intend that $m_1=m$ if $\nu = 0$ (or $V_2\equiv 0$) and $m_1=0$ if $V_1\equiv 0$. Thus, \eqref{lepd} becomes a Liouville type equation. Let us stress that $\Om_\e$ is an example of a non-simply connected domain. Without loss of generality, we shall assume in the rest of the paper that $\nu=1$, since we can replace $\nu V_2$ by $V_2$. However, we need the presence of $\nu$ when we compare \eqref{lepd} with equation \eqref{mfevi}. 

\medskip \noindent Finally, we point out some comments about the proof of the theorem. Following the main ideas presented in \cite{EFP}, we find a solution $u_\rho$ using a perturbative approach, precisely, we look for a solution of \eqref{lepd} as
\begin{equation}\label{ansatz}
u_\rho=U +\phi,
\end{equation}
where $U$ is a suitable ansatz built using the projection operator $P_\e$ onto $H^1_0(\Omega_\e)$(see \eqref{ePu}) and $U$ is defined as follows
$$U= \sum_{k=1}^{m_1}P_\e w_k-\frac1\tau\sum_{k=m_1+1}^{m} P_\e w_k\qquad \text{ with }\qquad
w_i(x)=\log\frac{2\alpha^2_i\delta_i^{\alpha_i}}{(\delta_i^{\alpha_i}+|x -\xi_i|^{\alpha_i})^2},$$
where $\de_i>0$, $i=1,\dots,m$ and $\al_i$'s are real parameters satisfying $\al_i> 2$ with $\al_i\not\in 2\mathbb N$ for all $i=1,\dots,m$; and $\phi\in H_0^1(\Om_\e)$ is a small remainder term. A careful choice of the parameters $\delta_j $'s and the radii $\e_j$'s is made in section \ref{sec2}  (see \eqref{choice1}) in order to make $U$ be a good approximated solution. Indeed, the error term $R$ given by
\begin{equation}\label{R} 
R= \Delta U+\rho(V_1 e^{U} - V_2e^{-\tau U} )
\end{equation}
is  small in $L^p$-norm for $p>1$ close to $1$ (see Lemma \ref{erre}). A linearization procedure around $U$ leads us to re-formulate \eqref{lepd} in terms of a nonlinear problem for $\phi$ (see equation \eqref{ephi}). We will prove the existence of such a solution $\phi$ to \eqref{ephi} by using a fixed point argument, thanks to some estimates in section \ref{sec4} (see \eqref{estnphi}). The corresponding solution $u_\rho$ in \eqref{ansatz} blows-up at the points $\xi_i$'s thanks to the asymptotic properties of its main order term $U$ (see \eqref{1519} in Corollary \ref{coro927}). In Section \ref{sec3} we will prove the invertibility of the linear operator naturally associated to the problem (see \eqref{ol}) stated in  Proposition \ref{elle}. 

%%%%%%%%%%%%%%%%%%%%%%%%%%%%%%%%%%%%%%%%%%%%%%%%%%%%%%%%%%%%%%%%%%%%%%%%%%%%%%%%%%%%%%
\section{The ansatz}\label{sec2}

\noindent Following the main ideas in \cite{EFP}, in this section we shall make a choice of the parameters $\de_j$'s in order to make $U$ a good approximation. Let $G(x,y)=-\frac{1}{2\pi}  \log |x-y|+H(x,y)$ be the Green function of $-\Delta$ in $\Omega$, where the regular part $H$ is a harmonic function in $\Omega$ so that $H(x,y)=\frac{1}{2\pi} \log |x-y|$ on $\fr\Om$. Let us introduce the coefficients $\be_{ij},$ $i,j=1,\dots,m,$ as the solution of the linear system
\begin{equation}\label{eqb2}
 {\be_{ij}}\left({1\over 2\pi}\log \e_j -H(\xi_j,\xi_j)\right)-\sum_{k\ne j}\be_{ik} G(\xi_j,\xi_k)=- 4\pi \al_i  H(\xi_i,\xi_j)+\left\{  \begin{array}{ll} 2 \al_i \log \delta_i &\hbox{if }i=j\\
2 \al_i  \log |\xi_i-\xi_j|& \hbox{if } i\not=j. \end{array} \right.
\end{equation}
Notice that \eqref{eqb2} can be re-written as the diagonally-dominant system
$${\be_{ij}} \log \e_j -2\pi \bigg[ \be_{ij} H(\xi_j,\xi_j)+ \sum_{k\ne j}\be_{ik} G(\xi_j,\xi_k) \bigg]=-8 \pi^2 \al_i H(\xi_i,\xi_j)+\left\{  \begin{array}{ll} 4\pi \al_i \log \delta_i &\hbox{if }i=j\\
4\pi \al_i  \log |\xi_i-\xi_j| & \hbox{if } i\not=j \end{array} \right.$$
for $\e_j$ small, which has a unique solution satisfying 
\begin{equation}\label{eqbij}
\beta_{ij}=\frac{4\pi \al_i \log \delta_i}{\log \e_j} \delta_{ij}+O(|\log \e_j|^{-1})
\end{equation}
where $\delta_{ij}$ is the Kronecker symbol. Introduce the projection $P_\e w$ as the unique solution of
\begin{equation}\label{ePu}
\left\{ \begin{array}{ll} 
\Delta P_\e w =\lap w &\text{in }\Om_\e\\
P_\e w=0,&\text{on }\fr\Om_\e.
\end{array}\right.
\end{equation} 
Notice that $w_i$ is a solution of the singular Liouville equation
\begin{equation*}
\left\{ \begin{array}{ll}\Delta w+|x -\xi_i|^{\alpha_i-2}e^{w}=0 &\text{in $\R^2$}\\
\ds\int_{\R^2} |x -\xi_i |^{\alpha_i-2} e^w dx<+\infty,& 
\end{array} \right.
\end{equation*}
By studying the harmonic function
$$\psi=P_\e w_i-w_i+\log\lf[2\al_i^2\de_i^{\al_i} \rg]-4\pi \al_i H(x,\xi_i)+\sum_{k=1}^m\be_{ik} G(x,\xi_k)$$
and using the maximum principle the following asymptotic expansion of $P_{\e}w_i$ was proved in \cite{EFP}.
\begin{lemma}\label{ewfxi}
There hold
\begin{eqnarray}\label{pui}
P_\e w_i = w_i-\log\lf[2 \al_i^2\de_i^{\al_i}\rg]+
4 \pi \al_i H(x,\xi_i)-\sum_{k=1}^m\be_{ik} G(x,\xi_k)+O \left( \de_i^{\al_i}+\Big(1+\frac{\log \delta_i }{\log \e_i}\Big)  \sum_{k=1}^m\e_k +\Big(\dfrac{ \e_i}{\de_i }\Big)^{\al_i }\right)
\end{eqnarray}
uniformly in $\Om_\e$ and 
\begin{eqnarray} \label{puii}
P_\e w_i=4 \pi \al_i G(x,\xi_i)-\sum_{k=1}^m\be_{ik} G(x,\xi_k)+O \left(\de_i^{\al_i}+\Big(1+\frac{\log \delta_i }{\log \e_i}\Big)  \sum_{k=1}^m\e_k +\Big(\dfrac{ \e_i}{\de_i }\Big)^{\al_i } \right)
\end{eqnarray}
locally uniformly in $\overline{\Om} \sm\{\xi_1,\dots,\xi_m\}$.
\end{lemma}

\noindent From the definition of $U$ and using \eqref{pui}-\eqref{puii}, we need to impose 
\begin{equation}\label{choice0.1}\left\{\begin{array}{ll}
\sum\limits_{j=1 }^{m_1} \beta_{ji} - \frac1\tau\sum\limits_{j=m_1+1 }^{m } \beta_{ji}=2\pi(\al_i-2) & i=1,\dots,m_1\\
 - \tau\sum\limits_{j=1 }^{m_1} \beta_{ji}+ \sum\limits_{j=m_1+1 }^{m } \beta_{ji}=2\pi(\al_i-2) & i=m_1+1,\dots, m,
\end{array} \right.
\end{equation}
see Corollary \ref{coro927}. Taking into account \eqref{eqbij},  \eqref{choice0.1} requires at main order that $\alpha_i \log \delta_i\simeq\frac{\alpha_i-2}{2} \log \e_i$, i.e. $\delta_i^{\alpha_i}\sim \e_i^{\frac{\alpha_i-2}{2}}$. Moreover, due to the presence of 
$\log\lf[2 \al_i^2\de_i^{\al_i}\rg]$ in \eqref{pui}-\eqref{puii} we need further to assume that the $\de_i^{\al_i}$'s have the same rate, as it is well known in elliptic problems with exponential nonlinearity in dimension two, see \cite{CL1,CL2,DEFM,dkm,EsFi,EGP,Fi}.

\medskip \noindent Summarizing, for any $i=1,\dots,m$ we choose
\begin{equation}\label{choice1}
\de_i^{\al_i}=d_i\rho , \quad  \e_i^{\al_i-2\over 2}=
r_i\rho,
\end{equation}
for the small parameter $\rho>0$, where $d_i,r_i$ will be specified below, and introduce 
$$\rho_i= \left\{\begin{array}{ll}
 (\al_i+2)H(\xi_i,\xi_i)+\sum\limits_{j=1\atop j\not=i }^{m_1}   (\al_j+2)G(\xi_i,\xi_j) - \frac{1}{\tau} \sum\limits_{j=m_1+1}^{m }   (\al_j+2)G(\xi_i,\xi_j) & i=1,\dots,m_1\\
(\al_i+2)H(\xi_i,\xi_i) -  \tau \sum\limits_{j=1 }^{m_1}   (\al_j+2)G(\xi_i,\xi_j)+  \sum\limits_{j=m_1+1\atop j\not=i}^{m }   (\al_j+2)G(\xi_i,\xi_j) & i=m_1+1,\dots, m. \end{array} \right.$$
Setting $A_i=\overline{B(\xi_i,\eta)} \setminus B(\xi_i,\e_i)$ for $\eta<\frac{1}{2} \min\lf\{ |\xi_i-\xi_j| : i\not= j\ ;  \ \hbox{dist}(\xi_i,\partial\Om):i= 1,\dots,m \rg\} $, by Lemma \ref{ewfxi} we deduce the following expansion.
\begin{corollary} \label{coro927}
Assume the validity of \eqref{choice0.1}. There hold
\begin{eqnarray} \label{1517}
U = w_i-\log\lf[2 \al_i^2\de_i^{\al_i}\rg]+
(\alpha_i-2) \log |x-\xi_i|+2 \pi \rho_i+O \left( \rho+ \sum_{k=1}^m\rho^{\frac{2}{\alpha_k-2}} +|x-\xi_i|\right)
\end{eqnarray}
uniformly in $A_i$, $i=1,\dots,m_1$, 
\begin{eqnarray} \label{1518}
-\tau U_= w_i- \log\lf[2 \al_i^2\de_i^{\al_i}\rg]+(\alpha_i-2) \log |x-\xi_i|+
2 \pi \rho_i+ O \left( \rho+ \sum_{k=1}^m\rho^{\frac{2}{\alpha_k-2}} +|x-\xi_i|\right)
\end{eqnarray}
uniformly in $A_i$, $i=m_1+1,\dots,m$, and 
\begin{eqnarray} \label{1519}
&& U=2 \pi \sum_{i=1}^{m_1} (\al_i+2) G(x,\xi_i)-\frac{2 \pi}{\tau} \sum_{i=m_1+1}^m (\al_i+2) G(x,\xi_i)+ O \left( \rho+ \sum_{k=1}^m\rho^{\frac{2}{\alpha_k-2}}  \right)
\end{eqnarray}
locally uniformly in $\overline{\Om} \sm\{\xi_1,\dots,\xi_m\}$.
\end{corollary}
\begin{proof}[\dem] The following expansions hold uniformly in $A_i$, $i=1,\dots,m_1$
\begin{equation*}
\begin{split}
U=&\, w_i-\log[2\al_i^2\de_i^{\al_i}] + 4\pi \al_i H(x,\xi_i)- \lf(\sum_{j=1}^{m_1} \beta_{ji}-\frac{1}{\tau} \sum_{j=m_1+1}^m \beta_{ji}\rg) G(x,\xi_i)-\sum\limits_{k=1\atop k\not=i}^{m } \beta_{ik}G(x,\xi_k)\\
&\, + \sum_{k=1\atop k\ne i}^{m_1}\bigg[4\pi\al_kG(x,\xi_k) -\sum\limits_{l=1\atop l\not=i}^{m } \beta_{kl}G(x,\xi_l)\bigg] -{1\over\tau} \sum_{k=m_1+1}^m\bigg[4\pi\al_kG(x,\xi_k) -\sum\limits_{l=1\atop l\not=i}^{m } \beta_{kl}G(x,\xi_l)\bigg] \\
&\,+ O \left( \rho+ \sum_{k=1}^m\rho^{\frac{2}{\alpha_k-2}}  \right).
\end{split}
\end{equation*}
Notice that
\begin{equation*}
\begin{split}
-&\sum_{k=1\atop k\ne i}^{m} \be_{ik}G(x,\xi_k)+\sum_{k=1\atop k\ne i}^{m_1} \bigg[ 4\pi \al_k  G(x,\xi_k) - \sum_{l=1\atop l\ne i}^{m} \be_{kl}G(x,\xi_l) \bigg]-{1\over\tau}\sum_{k=m_1}^{m}\bigg[4\pi\al_k G(x,\xi_k)- \sum_{l=1\atop l\ne i}^{m}\be_{kl} G(x,\xi_l)\bigg]\\
=&\sum_{j=1,j\ne i}^{m_1}\bigg[ 4\pi \al_j -\sum_{k=1}^{m_1} \be_{kj}+{1\over \tau} \sum_{k=m_1+1}^{m} \be_{kj} \bigg] G(x,\xi_j)- {1\over\tau} \sum_{j=m_1+1}^{m} \bigg[ 4\pi \al_j +\tau \sum_{k=1}^{m_1} \be_{kj}-\sum_{k=m_1+1}^{m} \be_{kj} \bigg] G(x,\xi_j). 
\end{split}
\end{equation*}
Hence, \eqref{1517} follows by using \eqref{choice0.1} and the choice of $\rho_i$. Similarly, we conclude \eqref{1518} and \eqref{1519}.
\end{proof}
\noindent As in \cite{EFP}, we obtain the validity of \eqref{choice0.1} by a suitable choice of $r_i$ and $d_i$. The following Lemma states a relation between $r_i$ and $d_i$, see \cite[Lemma 2.3]{EFP} for a proof.
\begin{lemma}\label{lem909} If $r_i=d_i e^{-\pi \rho_i}$ for all $i=1,\dots,m$, then \eqref{choice0.1} does hold.
\end{lemma}
\noindent Finally, we need that $-\Delta P_\e w_i=|x-\xi_i|^{\alpha_i-2} e^{w_i}$ match with $\rho V_1 e^{U}$ and $\rho \tau V_2 e^{-\tau U}$ in $A_i$ for $i=1,\dots,m_1$ and for $i=m_1+1,\dots,m$, respectively. As we will see below, this is achieved by requiring that 
\begin{equation}\label{choice3}
2 \al_j^2\de_j^{\al_j}  =\rho  V_1(\xi_j) e^{2\pi \rho_j } \quad\text{and}\quad
2 \al_j^2\de_j^{\al_j}  =\rho \tau  V_2(\xi_j) e^{2\pi \rho_j }
\end{equation}
for $i=1,\dots,m_1$ and for $i=m_1+1,\dots,m$, respectively. The choice 
\begin{equation}\label{choiced}
d_i=\left\{ \begin{array}{ll}
\frac{V_1(\xi_i) e^{2\pi  \rho_i}}{2\alpha_i^2} &i=1,\dots,m_1\\
\frac{V_2(\xi_i) e^{2\pi  \rho_i}}{2\alpha_i^2\tau } &i=m_1+1,\dots,m
\end{array} , \right. \quad r_i=\left\{ \begin{array}{ll}
\frac{V_1(\xi_i) e^{\pi \rho_i}}{2\alpha_i^2} &i=1,\dots,m_1\\
\frac{V_2(\xi_i) e^{\pi \rho_i}}{2\alpha_i^2\tau } &i=m_1+1,\dots,m
\end{array} \right.
\end{equation}
guarantees the validity of \eqref{choice0.1} and \eqref{choice3} in view of Lemma \ref{lem909}. We are now ready to estimate the precision of our ansatz $U$.
\begin{lemma}\label{erre}
There exists $\rho_0>0,$ $p_0>1$ and  $C>0$ such that for any $\rho\in(0,\rho_0)$ and $p\in(1,p_0)$
\begin{equation}\label{re1}
\|R\|_p\le  C \rho^{\sigma_p}
\end{equation}
for some $\sigma_p>0$.
\end{lemma}
\begin{proof}[\dem] First, note that
$$\Delta U=\left\{ \begin{array}{ll} -|x-\xi_i|^{\alpha_i-2} e^{w_i}+O(\rho)&\hbox{in }A_i, \ i=1,\dots,m,\\[0.2cm] 
\ds \frac{ 1}{\tau} |x-\xi_i|^{\alpha_i-2}e^{w_i}+O(\rho)& \hbox{in }A_i,\ i=m_1+1,\dots,m\\
O(\rho) &\hbox{in }\Omega_\e \setminus \ds \bigcup_{i=1}^m A_i \end{array} \right.$$
in view of  \eqref{choice1}. 
By \eqref{1517} we have that
\begin{eqnarray}\label{214}
V_1 e^{U} =
\frac{V_1(\xi_i) e^{2 \pi \rho_i}}{2 \al_i^2\de_i^{\al_i}} |x-\xi_i|^{\alpha_i-2} e^{w_i}\lf[1+ O\bigg( \rho+ \sum_{k=1}^m\rho^{\frac{2}{\alpha_k-2}} +|x-\xi_i| \bigg)\rg] \label{1036}
\end{eqnarray}
uniformly in $A_i$ for any $i=1,\dots,m_1$ and by \eqref{1518}
\begin{eqnarray} \label{1503} V_1 e^{U} =
O\lf(\Big[\frac{|x-\xi_i|^{\alpha_i-2}}{\delta_i^{\alpha_i}} e^{w_i}\Big]^{-\frac{1}{\tau}}\rg)\quad \text{uniformly in $A_i$, for $i=m_1+1,\dots,m$.}  
\end{eqnarray} 
 Similarly, by \eqref{1518} we have that
 \begin{eqnarray} \label{1037}
V_2 e^{-\tau U} = \frac{V_2(\xi_i) e^{2 \pi \rho_i}}{2 \al_i^2\de_i^{\al_i}} |x-\xi_i|^{\alpha_i-2} e^{w_i}\lf[1+ O\bigg( \rho+ \sum_{k=1}^m\rho^{\frac{2}{\alpha_k-2}} +|x-\xi_i| \bigg)\rg] 
\end{eqnarray}
uniformly in $A_i$ for any $i=m_1+1,\dots,m$ and by \eqref{1517}
\begin{eqnarray}\label{1502}
V_2 e^{-\tau U} =
O\lf(\Big[\frac{|x-\xi_i|^{\alpha_i-2}}{\delta_i^{\alpha_i}} e^{w_i}\Big]^{-\tau}\rg) \quad \text{ uniformly in $A_i$, for $i=1,\dots,m_1$.}
\end{eqnarray}
Hence, by using \eqref{choice1}, \eqref{choice3}, \eqref{choiced} and \eqref{214}-\eqref{1502} we can estimate the error term $R$ as:
 \begin{eqnarray} \label{ex5}
R =|x-\xi_i|^{\al_i-2}e^{w_i} O\Big(\rho^\sigma+|x-\xi_i|\Big) +O(\rho^\sigma) 
\end{eqnarray}
in $A_i$, $i=1,\dots,m_1$, and 
\begin{eqnarray} \label{ex6.1}
R =- \frac{1}{\tau}  |x-\xi_i|^{\alpha_i-2} e^{w_i} O\Big(\rho^\sigma +|x-\xi_i|\Big)+O(\rho^\sigma) 
\end{eqnarray}
in $A_i$, $i=m_1+1,\dots,m$, where $\sigma=\min\{{1\over \al_i} \ : \ i=1,\dots,m \}$, while $ R=O(\rho)$ does hold in $\Omega_\e \setminus \ds \bigcup_{i=1}^m A_i $. By \eqref{ex5}-\eqref{ex6.1} we finally get that there exist $\rho_0>0$ small, $p_0>1$ close to $1$ so that $\|R\|_p=O(\rho^{\sigma_p})$ for all $0<\rho\leq \rho_0$ and $1<p\leq p_0$, for some $\sigma_p>0$.
\end{proof}

\section{The nonlinear problem and proof of main result}\label{sec4}

In this section we shall study the existence of a function $\phi_\rho \in H^1_0(\Om_\e)$,  a small remainder term which satisfies the following nonlinear problem:
\begin{equation}\label{ephi}
\left\{ \begin{array}{ll}
L(\phi)=-[R+N(\phi)] 
& \text{in } \Om_\e\\
\phi=0, &\text{on }\fr\Om_\e,
\end{array} \right.
\end{equation}
where the linear operator $L$ is defined as
\begin{equation}\label{ol}
L(\phi) = \Delta \phi + W(x)\phi,\qquad W(x):=\rho V_1(x) e^{U} + \rho  \tau V_2(x) e^{-\tau U}
\end{equation}
and the nonlinear term $N(\phi) $ is given by
\begin{equation}\label{nlt}
\begin{aligned}
N(\phi)= &\rho V_1(x)e^{U}\big(e^{ \phi}- \phi - 1\big)-\rho  V_2(x) e^{-\tau U}\big( e^{ -\tau \phi}+ \tau  \phi  -1 \big)\\ 
\end{aligned}
\end{equation}
It is readily checked that $\phi$ is a solution to \eqref{ephi} if and only if $u_\e$ given by \eqref{ansatz} is a solution to \eqref{lepd}. In section \ref{sec3} we will prove the following result.

\begin{proposition}\label{elle}
For any $p>1,$ there exists $\rho_0>0 $ and  $C>0$ such that for any $\rho\in(0,\rho_0)$ and $h \in L^p(\Om_\e)$ there exists a unique $\phi\in H^1_0(\Om_\e)$ solution of
\begin{equation}\label{pl}
L(\phi)=h \ \hbox{ in }\ \Om_\e,\ \ \ \phi=0\ \hbox{ on }\ \partial\Omega_\e,
\end{equation}
which satisfies
\begin{equation}\label{estphi}
\|\phi\|\le C|\log\rho | \ \|h\|_p.
\end{equation}
\end{proposition}

The latter proposition implies that the unique solution $\phi=T(h)$ of \eqref{pl} defines a continuous linear map from $L^p(\Om_\e)$ into $H_0^1(\Om_\e)$, with norm bounded by $C|\log\rho|$. Concerning problem \eqref{ephi} we have the following fact.

\begin{proposition}\label{p3}
There exist $p_0>1$ and $\rho_0>0$ so that for any $1<p<p_0$ and 
all $0<\rho \leq \rho_0$, the problem \eqref{ephi} admits a unique solution $\phi(\rho) \in H_0^1(\Om_\e)$, where $ R$, $L$ and $ N$ are given by \eqref{R}, \eqref{ol} and \eqref{nlt}, respectively. Moreover, there is a constant $C>0$ such that
$$\|\phi\|_\infty\le C\rho^{\sigma_p} |\log \rho |,$$
for some $\sigma_p>0$.
\end{proposition}

Here, $\sigma_p$ is the same as in \eqref{re1}. We shall use the following estimate.

\medskip

\begin{lemma}
There exist $p_0>1$ and $\rho_0>0$ so that for any $1<p<p_0$ and 
all $0<\rho\leq \rho_0$ it holds
\begin{equation}\label{estnphi}
\|  N (\phi_1)- N(\phi_2) \|_p  \le C\rho^{\sigma_p'} \|\phi_1-\phi_2\|
\end{equation}
for all $\phi_i\in H_0^1(\Om_\e)$ with $\|\phi_i\|\le M \rho^{\sigma_p}|\log\rho|$, $i=1,2$, and for some $\sigma_p'>0$. In particular, we have that
\begin{equation}\label{estnphi1}
\| N (\phi) \|_p  \le C\rho^{\sigma_p'} \ \|\phi\|
\end{equation}
for all $\phi\in H_0^1(\Om_\e)$ with $\|\phi\|\le M \rho^{\sigma_p}|\log\rho|$.
\end{lemma}

\begin{proof}[\dem] We will argue in the same way as in \cite[Lemma 5.1]{op}, see also \cite[Lemma 3.4]{EFP}. First, we point out that
\begin{equation*}
N(\phi)=\sum_{i=1}^2\rho \lf\{f_i(\phi)-f_i(0)-f'_i(0)[\phi] \rg\},\ \ \ \text{where} \ \ \
f_i(\phi)= V_i(x) e^{(-\tau)^{i-1}(U+\phi)}.
\end{equation*}
Hence, by the mean value theorem we get that
\begin{equation}\label{mvtn}
\begin{split}
N(\phi_1)-N(\phi_2) 
&\,=\sum_{i=1}^2\rho \lf\{f_i'(\phi_{\theta_i} )-f_i'(0) \rg\} [\phi_1-\phi_2] =\sum_{i=1}^2\rho  f_i''(\ti \phi_{\mu_i}) [\phi_{\theta_i}, \phi_1-\phi_2],
\end{split}
\end{equation}
where $\phi_{\theta_i}=\theta_i\phi_1+(1-\theta_i)\phi_2$, $\ti\phi_{\mu_i}=\mu_i\phi_{\theta_i}$ for some $\theta_i,\mu_i\in[0,1]$, $i=1,2$, and \linebreak $f_i''(\phi)[\psi,v]=\tau^{2(i-1)} V_i(x)e^{ (-\tau)^{i-1}(U+\phi) } \psi v$. Using H\"older's inequalities we get that
\begin{equation}\label{hin}
\lf\| f_i''(\phi)[\psi,v]\rg\|_p\le  |\tau|^{2(i-1)} \| V_ie^{ (-\tau)^{i-1}(U+\phi)} \|_{pr_i}  \|\psi\|_{ps_i} \|v\|_{pt_i}
\end{equation}
with $\ds {1\over r_i} +{1\over s_i} + {1\over t_i}=1$. We have used the H\" older's inequality and $\ds \|uvw\|_q\le \|u\|_{qr}\|v\|_{qs}\|w\|_{qt}$ with $\ds{1\over r }+{1\over s }+{1\over t}=1$. Now, let us estimate $\ds  \| V_ie^{ (-\tau)^{i-1}(U+\phi)} \|_{pr_i} $ with $\phi=\ti\phi_{\mu_i}$, $i=1,2$. By \eqref{1517} and \eqref{214} and the change of variable $x=\delta_i y+\xi_i$ let us estimate
\begin{eqnarray}
\int_{A_i} \lf| V_1 e^{ U} \rg|^q dx
&=& O\lf( \delta_i^{2-(\al_i+2)q}  \int_{\frac{\e_i}{\delta_i}\leq |y|\leq \frac{\eta}{\delta_i} } \lf| \frac{|y|^{\alpha_i-2}}{(1+|y|^{\alpha_i})^2 } \rg|^q \lf[1+  O\bigg( \rho+ \sum_{k=1}^m\rho^{\frac{2}{\alpha_k-2}} +\delta_i |y| \bigg)\rg]^q dy \rg)\nonumber \\
&=& O\lf( \delta_i^{2-(\al_i+2)q} \rg) \label{1036}
\end{eqnarray}
for any $i=1,\dots,m_1$ and similarly, by \eqref{1518} and \eqref{1037} we get that
\begin{eqnarray} \label{10372}
\int_{A_i} \lf|V_2 e^{-\tau U} \rg|^qdx=O\lf( \delta_i^{2-(\al_i+2)q}  \rg)
\end{eqnarray}
for any $i=m_1+1,\dots,m$, in view of \eqref{choice1} and \eqref{choiced}. By \eqref{choice1}, \eqref{1518} and \eqref{1503} we get the estimate
\begin{eqnarray}
\int_{A_i} \lf|V_1 e^{ U} \rg|^qdx&=&
O\bigg(\de_i^{\al_iq\over \tau} \int_{A_i}  \Big[|x-\xi_i|^{\alpha_i-2} e^{U_i}\Big]^{-{q\over\tau} } dx \bigg) = O\bigg(\de_i^{{(\al_i+2)q \over \tau} +2} \int\limits_{\frac{\e_i}{\delta_i}\leq |y|\leq \frac{\eta}{\delta_i} } \Big[\frac{|y|^{\alpha_i-2}}{(1+|y|^{\alpha_i})^2 }  \Big]^{-{q\over \tau} } dy\bigg)\nonumber\\
&=&  O\bigg(\de_i^{{(\al_i +2)q \over \tau} +2} \bigg[\int_{\frac{\e_i}{\delta_i}}^1  s^{1-{(\al_i-2)q\over\tau} }\,ds + \int_1^\frac{\eta}{\delta_i}  s^{1+{(\al_i+2)q\over\tau} }\,ds\bigg]\bigg) = O(1) \label{0054}
\end{eqnarray}
for all $i=m_1+1,\dots,m$. Similarly, by \eqref{choice1}-\eqref{1517} we deduce that
\begin{eqnarray} \label{0107}
\int_{A_i} \lf| V_2 e^{-\tau  U} \rg|^q dx=O(1)
\end{eqnarray}
for $i=1,\dots,m_1$, in view of \eqref{1502}. Therefore, by using \eqref{1519}, \eqref{1036} and \eqref{0054} we deduce that
$$\lf\| V_1e^{U  } \rg\|_{q}^q=\sum_{i=1}^{m_1} O\lf( \de_i^{2  - (\al_i+2)q } \rg)=\sum_{i=1}^{m_1} O\lf( \rho^{ {2-(\al_i+2)q\over \al_i} }\rg) \quad\text{for any } q\ge 1$$
and, by using \eqref{1519}, \eqref{10372} and \eqref{0107} we obtain that 
$$\lf\| V_2e^{ -\tau U } \rg\|_{q}^q=\sum_{i=m_1+1}^{m} O\lf( \rho^{ {2-(\al_i+2)q\over \al_i} }\rg)\quad\text{for any } q\ge 1.$$
 
On the other hand, using the estimate $|e^a-1|\le |a|$ uniformly for any $a$ in compact subsets of $\R$ we have that
\begin{equation*}
\begin{split}
\lf\| V_ie^{(-\tau)^{i-1}(U +\ti\phi_{\mu_i})}- V_ie^{(-\tau)^{i-1} U }\rg\|_q&=\lf(\int_{\Om_\e} \lf| V_i e^{(-\tau)^{i-1}U} \rg|^q \ \lf| e^{(-\tau)^{i-1}\ti\phi_{\mu_i}}- 1  \rg|^q \rg)^{1/q}\\
&\le \lf\| V_ie^{(-\tau)^{i-1}U }\rg\|_{qs_i'}\lf\| \ti\phi_{\mu_i} \rg\|_{qt_i'},
\end{split}
\end{equation*}
with $\ds{1\over s_i'}+{1\over t_i'}=1$, $i=1,2$, in view of $\|\ti\phi_{\mu_i}\|\le  \rho^{\sigma_p}|\log \rho|\le C$, $i=1,2$. Hence, it follows that
$$\lf\| V_ie^{(-\tau)^{i-1}(U+\ti\phi_{\mu_i} ) }- V_ie^{(-\tau )^{i-1} U }\rg\|_q\le C \rho^{ \sigma_{0,qs_i'} -1 } \lf\| \ti\phi_{\mu_i} \rg\|\le C \rho^{ \sigma_{0,qs_i'} -1}\ \rho^{\sigma_p}|\log \rho|,$$
where for $q>1$ we denote $\ds \sigma_{0, q } = \min\lf\{ {2- 2  q\over \al_j q }\ :\ j=1,\dots,m\rg\}$, in view of $\|\ti\phi_{\mu_i}\|\le  \rho^{\sigma_p}|\log \rho|$, $i=1,2$. Note that $\ds \sigma_{0,q} -1\le {2-(\al_i+2)q\over \al_i q}\quad\text{ for any} \quad i=1,\dots,m$ and $\sigma_{0,q}<0$ for any $q>1$. By the previous estimates we find that 
\begin{equation}\label{ev}
 \lf\| V_i e^{(-\tau)^{i-1} (U+\ti\phi_{\mu_i} )}\rg\|_{q}=O\lf( \rho^{ \sigma_{0,qs_i'} -1 + \sigma_p}|\log \rho|+ \rho^{ \sigma_{0,q} -1 } \rg).
 \end{equation}
Also, choosing $q$ and $s_i'$, $i=1,2$, close enough to 1, we get that  $0< \sigma_p +\sigma_{0,qs_i'}$. Now, we can conclude the estimate by using \eqref{mvtn}-\eqref{ev} to get
\begin{equation*}
\begin{split}
\|N(\phi_1)-N(\phi_2)\|_p &\,\le \sum_{i=1}^2\rho \lf\| f_i''(\ti \phi_{\mu_i}) [\phi_{\theta_i}, \phi_1-\phi_2]\rg\|_p  \le\,C \sum_{i=1}^2\rho  \|V_i e^{ (- \tau)^{i-1}(U+\ti\phi_{\mu_i} )}\|_{p r_i}  \| \phi_{\theta_i} \| \| \phi_1-\phi_2\|\\
&\, \le\,C \sum_{i=1}^2 \rho^{\sigma_p +\sigma_{0,pr_i} }  | \log \rho| \| \phi_1-\phi_2\|
 \le\,C \rho^{\sigma_p' }  \|\phi_1-\phi_2\|,
\end{split}
\end{equation*}
where $\sigma'_p={1\over 2}\min\{ \sigma_p + \sigma_{0,pr_i}  \ : \ i=1,2\} >0$ choosing $r_i$ close to 1 so that $\sigma_p+\sigma_{0,pr_i} >0$ for $i=1,2$. Let us stress that $p>1$ is chosen so that $\sigma_p>0$ and $\sigma_p'>0$.
\end{proof}

We are now in position to study the nonlinear problem \eqref{ephi} and to prove our main result Theorem \ref{main}.

\begin{proof}[\bf Proof of the Proposition \ref{p3}] Notice that from Proposition \ref{elle} problem \eqref{ephi} becomes
$$\phi=-T(R +N(\phi)):=\ml{A}(\phi).$$
For a given number $M>0$, let us consider $
\ml{F}_M = \{\phi\in H : \| \phi \| \le
M \rho^{\sigma_p} |\log\rho|\}$. 
From the Proposition \ref{elle}, \eqref{re1} and \eqref{estnphi1}, we get for any $\phi\in \ml{F}_M$,
\begin{equation*}
\begin{split}
\|\ml{A}(\phi)\| & \le C | \log \rho |\lf[ \| R \|_p+ \|N(\phi)\|_p\rg] 
\le C |\log \rho |\lf[ \rho^{\sigma_p}+\rho^{\sigma_p'} \|\phi\|\rg]\\
& \le C \rho^{\sigma_p} | \log \rho |\lf[1 + M\rho^{ \sigma_p' } |\log\rho|\rg].
\end{split}
\end{equation*}
Given any $\phi_1,\phi_2\in\ml{F}_M$, we have that $\ml{A}(\phi_1)-\ml{A}(\phi_2) = - T\lf(N(\phi_1)
- N(\phi_2)\rg)$ and
\begin{equation*}
\begin{split}
\|\ml{A}(\phi_1)-\ml{A}(\phi_2)\| & \le C |\log
\rho | \lf\| N(\phi_1) - N(\phi_2)\rg\|_p
\le C\rho^{\sigma_p'} |\log\rho| \ \|\phi_1-\phi_2\|,
\end{split}
\end{equation*}
with $C$ independent of $M$, by using Proposition \ref{elle} and \eqref{estnphi}. Therefore, for some $\sigma>0$ we get that $\|\ml{A}(\phi_1)-\ml{A}(\phi_2)\|  \le C \rho^{\sigma} |\log \rho| \|\phi_1-\phi_2\|$. It follows that for all $\rho$ sufficiently small $\ml{A}$ is a
contraction mapping of $\ml{F}_M$ (for $M$ large enough), and
therefore a unique fixed point of $\ml{A}$ exists in $\ml{F}_M$.
\end{proof}

\begin{proof}[\bf Proof of the Theorem \ref{main}] Taking into account \eqref{ansatz} and the definition of $U$, the existence of a solution 
$$u_\rho=\sum_{j=1}^{m_1}P_\e w_j - \dfrac{1}{\tau } \sum_{j=m_1+1}^m P_\e w_j + \phi$$ 
to equation \eqref{lepd} follows directly by Proposition \ref{p3}. The asymptotic behavior of $u_\rho$ as $\rho\to 0^+$ follows from \eqref{1519} in Lemma \ref{ewfxi} and estimate for $\phi$ in Proposition \ref{p3}. Precisely, we have that  as $\rho\to 0^+$
$$u_\rho=2 \pi \sum_{i=1}^{m_1} (\al_i+2) G(\cdot,\xi_i)-\frac{2 \pi}{\tau} \sum_{i=m_1+1}^m (\al_i+2) G(\cdot,\xi_i)+ o(1)$$
locally uniformly in $\bar\Om \sm\{\xi_1,\dots,\xi_m\}$. Therefore, from the behavior of Green's function we conclude that $u_\rho$ blows-up positively at $\xi_1,\dots,\xi_{m_1}$ and negatively at $\xi_{m_1+1},\dots,\xi_m$ as $\rho$ goes to zero.
\end{proof}

%%%%%%%%%%%%%%%%%%%%%%%%%%%%%%%%%%%%%%%%%
%%%%%%%%%%%%%%%%%%%%%%%%%%%%%%%%%%%%%%%%%

\section{The linear theory}\label{sec3}

In this section, following ideas presented in \cite[Section 4]{EFP} we present the invertibility of the linear operator $L$ defined in \eqref{ol}. Roughly speaking in the scale annulus $\de_i^{-1} (B_i-\xi_i)$ the operator $L$ apporaches to the following linear operator in $\R^2$
$$L_i(\phi)=\Delta\phi+{2\al_i^2|y|^{\al_i-2}\over (1+|y|^{\al_i})^2}\phi,\qquad i=1,\dots, m.$$
It is well known that the bounded solutions of $L_i(\phi)=0$ in
$\R^2$ are precisely linear combinations of the functions
\begin{equation*}
Y_{1i}(y) = { |y|^{\al_i\over 2} \over 1+|y|^{\al_i}}\cos\Big({\al_i\over 2}\theta\Big),\quad Y_{2i}(y) = { |y|^{\al_i\over 2} \over 1+|y|^{\al_i}}\sin\Big({\al_i\over 2}\theta\Big)\quad\text{and}\quad Y_{0i}(y) = {1-|y|^{\al_i}\over 1+|y|^{\al_i}},
\end{equation*}
which are written in polar coordinates for $i=1,\dots, m$. See \cite{DEM5} for a proof. In our case, we will consider solutions of $L_i(\phi)=0$ such that $\int_{\R^2}|\nabla\phi(y)|^2\, dy<+\infty$. See \cite[Theorem A.1]{op} for a proof. 

Let us introduce the following Banach spaces for $j=1,2$
\begin{equation}\label{lai}
L_{\al_i}(\R^2)=\lf\{u\in W_{\text{loc} }^{1,2}(\R^2)\ :\ \int_{\R^2}{|y|^{\al_i-2}\over (1+|y|^{\al_i})^2}|u(y)|^2\, dy<+\infty\rg\}
\end{equation}
and
\begin{equation}\label{hai}
H_{\al_i}(\R^2)=\lf\{u\in W_{\text{loc} }^{1,2}(\R^2)\ :\ \int_{\R^2}|\nabla u(y)|^2\, dy+\int_{\R^2}{|y|^{\al_i-2}\over (1+|y|^{\al_i})^2}|u(y)|^2\, dy<+\infty\rg\}
\end{equation}
endowed with the norms
$$\|u\|_{L_{\al_i}}:=\lf(\int_{\R^2}{|y|^{\al_i-2}\over (1+|y|^{\al_i})^2}|u(y)|^2\, dy\rg)^{1/2}$$
and
$$\|u\|_{H_{\al_i}}:=\lf(\int_{\R^2} |\nabla u(y)|^{2}\, dy+\int_{\R^2}{|y|^{\al_i-2}\over (1+|y|^{\al_i})^2}|u(y)|^2\, dy\rg)^{1/2}.$$
It is important to point out the compactness of the embedding $i_{\al_i}:H_{\al_i}(\R^2)\to L_{\al_i}(\R^2)$, (see for example \cite{GP}).

\begin{proof}[\bf Proof of the Proposition \ref{elle}]
The proof will be done in several steps. Let us assume the opposite, namely, the existence of $p>1$, sequences $\rho_n\to0$, $\ve_n:=\e(\rho_n)\to0$, functions $h_n\in L^p(\Om_{\ve_n})$, $\phi_n\in W^{2,2}(\Om_{\ve_n})$ such that
\begin{equation}\label{eqphin}
L(\phi_n)=h_n\ \ \text{in}\ \ \Om_{\ve_n},\ \ \phi_n=0\ \ \text{on}\ \ \fr\Om_{\ve_n}
\end{equation}
$\|\phi_n\|=1$ and $|\log\rho_n|\ \|h_n\|_p=o(1)$ as $n\to+\infty$. We will shall omit the subscript $n$ in $\de_{i,n} =\de_i$. Recall that $\de_i^{\al_i}=d_{i,n}\rho_n$ and points $\xi_1,\dots,\xi_m\in\Om$ are fixed.

Now, define $\Phi_{i,n}(y):=\phi_{n}(\xi_i+\de_i y)$ for $y\in\Om_{i,n}:=\de_i^{-1}(\Om_{\ve_n}-\xi_i)$, $i=1,\dots,m$. Thus, extending $\phi_n=0$ in $\R^2\sm\Om_{\ve_n}$ we can prove the following fact.

\begin{claim}\label{claim2}
There holds that the sequence $\Phi_{i,n}$ converges (up to subsequence) to $\Phi_i^*=a_jY_{0i}$ for $i=1,\dots,m$, weakly in $H_{\al_i}(\R^2)$ and strongly in $L_{\al_i}(\R^2)$ as $n\to+\infty$ for some constant $a_{i}\in\R$, $i=1,\dots,m$.
\end{claim}

\begin{proof}[\dem]
First, we shall show that the sequence $\{\Phi_{i,n}\}_n$ is bounded in $H_{\al_i}(\R^2)$. Notice that for $i=1,\dots,m$
$$\|\Phi_{i,n}\|_{H^1_0(\Omega_{i,n})}=\int_{\Om_{i,n}}\de_i^2|\grad \phi_{i,n}(\xi_i+\de_i y)|^2\,dy=\int_{\Om_{\ve_n} }|\grad\phi_n(x)|^2\, dx=1.$$ 
Thus, we want to prove that there is a constant $M>0$ such for all $n$ (up to a subsequence)
$$\|\Phi_{i,n}\|_{L_{\al_i}}^2=\int_{\Omega_{i,n} } {|y|^{\al_i-2}\over (1+|y|^{\al_i})^2} \Phi_{i,n}^2(y)\, dy\le M.$$
Notice that for any $i\in\{1,\dots,m\}$ we find that
 in $\Om_{i,n}$
\begin{equation}\label{eqPhi1}
\begin{split}
\lap \Phi_{i,n}&+\de_i^2W (\xi_i+\de_i y) \Phi_{i,n} =\de_i^2h_n(\xi_i+\de_i y).
\end{split}
\end{equation}
Furthermore, it follows that $\Phi_{i,n}\to\Phi_i^*$ weakly in $H^1_0(\Omega_{i,n})$ and strongly in $L^p(K)$ for any $K$   compact sets in $\mathbb R^2$.
Now, we multiply \eqref{eqPhi1} by $\Phi_{i,n}$ for any $i\in\{1,\dots,m\}$ and we get 
\begin{equation*}
\begin{split}
-\int_{\Omega_{i,n} }&|\nabla \Phi_{i,n}|^2 + \int_{\Omega_{i,n} } \de_i^2W(\xi_i+\de_i y)  \Phi_{i,n}^2=\,\int_{\Omega_{i,n} }\de_i^2h_n(\xi_i+\de_i y) \Phi_{i,n}.
\end{split}
\end{equation*}
Using \eqref{choice1}, \eqref{1517}, \eqref{1518} and \eqref{choiced}, we obtain that
\begin{equation}\label{dik1}
\de_i^2W(\xi_i+\de_i y)=\ds{2\al_i^2|y|^{\al_i-2}\over (1+|y|^{\al_i})^2}+O(\de_i^2\rho)\quad \text{ for all }i=1,\dots,m
\end{equation}
uniformly for $y$ on compact subsets of $\R^2$.
Thus, we deduce that
\begin{equation}\label{npsi}
\sum_{i=1}^{m} 2\al_i^2\| \Phi_{i,n}\|_{L_{\al_i}}^2
=\,1 + o(1)
\end{equation}
in view of 
$$\int_{\Om_{i,n}} \de_i^2W(\xi_i+\de_i y )\Phi_{i,n}^2\, dy=\int_{\Om_{\e_n}}W\phi^2_n\, dx=\sum_{i=1}^{m} 2\al_i^2\| \Phi_{i,n}\|_{L_{\al_i}}^2 + o(1)$$
since
$$\int_{\Om_{i,n} } {2\al_i^2|y|^{\al_i-2}\over (1+|y|^{\al_i})^2} \Phi_{i,n}^2(y)\, dy=2\al_i^2\|\Phi_{i,n}\|_{L_{\al_i} }^2,\ \ \text{for }i=1,\dots,m.$$
Therefore, the sequence $\{\Phi_{i,n}\}_n$ is bounded in $H_{\al_i}(\R^2)$, so that there is a subsequence $\{\Phi_{i,n}\}_n$ and functions $\Phi_i^*$, $i=1,2$ such that $\{\Phi_{i,n}\}_n$ converges to $\Phi_i^*$ weakly in $H_{\al_i}(\R^2)$ and strongly in $L_{\al_i}(\R^2)$. 
Furthermore, we have that
$$\int_{\Om_{i,n}} (\de_i^2|h_n(\xi_i+\de_i y)|)^p\, dy=\de_i^{2p-2}\int_{\Om_{\e_n} }|h_n(x)|^p\, dx=\de_i^{2p-2}\|h_n\|_p^p=o(1).$$
Hence, taking into account \eqref{eqPhi1}-\eqref{dik1} we deduce that $\Phi^*_i$ a solution to 
$$\lap\Phi+{2\al_i^2|y|^{\al_i-2}\over(1+|y|^{\al_i})^2}\Phi=0,\qquad i=1,\dots,m,\qquad  \text{in $\R^2\sm\{0\}$}.$$
It is standard that $\Phi_i^*$, $i=1,\dots,m$ extend to a solution in the whole $\R^2$. Hence, by using symmetry assumptions if necessary, we get that $\Phi^*_i=a_iY_{0i}$ for some constant $a_{i}\in\R$, $i=1,\dots,m$.
\end{proof}

For the next step we construct some suitable test functions. To this aim, introduce the coefficients $\gamma_{ij}$'s and $\ti\gamma_{ij}$'s, $i,j=1,\dots,m$, as the solution of the linear systems
\begin{equation}\label{gamaij}
\gamma_{ij}\lf[-{1\over 2\pi}\log\e_i+H(\xi_i,\xi_i)\rg]+\sum_{k=1,k\ne i}^m \gamma_{kj} G(\xi_k,\xi_i)=2\de_{ij}
\end{equation}
and
\begin{equation}\label{gamatij}
\ti\gamma_{ij}\lf[-{1\over 2\pi}\log\e_i+H(\xi_i,\xi_i)\rg]+\sum_{k=1,k\ne i}^m \ti \gamma_{kj} G(\xi_k,\xi_i)=\begin{cases}
\ds {4\over 3} \al_j\log\de_j +{8\over 3} + {8\pi\over 3}\al_j H(\xi_j,\xi_j)& \text{if } i=j\\[0.3cm]
\ds {8\pi\over 3}\al_j G(\xi_i,\xi_j),& \text{if } i\ne j,
\end{cases}
\end{equation}
respectively. Notice that both systems \eqref{gamaij} and \eqref{gamatij} are diagonally dominant, system \eqref{gamaij} has solutions
$$\gamma_{ij}=
 -{4\pi\over \log\e_j}\de_{ij}+O\Big({1\over |\log\rho|^2}\Big)=-{2\pi(\al_j-2)\over \log \rho} \de_{ij}+O\Big({1\over |\log\rho|^2}\Big)$$
and for the system \eqref{gamatij} we get
$$\ti\gamma_{ij}=\ds -{8\pi\over 3}{\al_j\log\de_j\over\log\e_j}\de_{ij}+O\Big({1\over |\log\rho|}\Big)=-{4\pi\over 3}(\al_j-2)\de_{ij}+O\Big({1\over |\log\rho|}\Big)$$
where $\delta_{ij}$ is the Kronecker symbol. Here, we have used \eqref{choice1}. Consider now for any $j\in\{1,\dots, m\}$ the functions $\ds \eta_{0j}(x)= - \dfrac{2\de_j^{\al_j} }{\de_j^{\al_j}+|x-\xi_j|^{\al_j}}$
and
$$\eta_j(x)={4\over 3}\log(\de_j^{\al_j}+|x-\xi_j|^{\al_j}){\de_j^{\al_j}-|x-\xi_j|^{\al_j}\over \de_j^{\al_j}+|x-\xi_j|^{\al_j} }+ {8\over 3}{\de_j^{\al_j} \over \de_j^{\al_j}+|x-\xi_j|^{\al_j} },$$
 so that 
 $$\lap \eta_{0j} + |x-\xi_j|^{\al_j-2}e^{w_j} \eta_{0j}=-|x-\xi_j|^{\al_j}e^{w_j}\quad\text{ and }\quad\lap\eta_j+|x - \xi_j|^{\al_j-2}e^{w_j}\eta_j=|x-\xi_j|^{\al_j-2}e^{w_j}Z_{0j},$$
where $Z_{0j}(x)=Y_{0j}(\de_j^{-1} [x-\xi_j])=\dfrac{\de_j^{\al_j}-|x-\xi_j|^{\al_j}}{\de_j^{\al_j}+|x-\xi_j|^{\al_j}}$. Notice that $\eta_{0j} + 1= -Z_{0j}$ and, by similar arguments as to obtain expansion \eqref{pui}, by studying the harmonic functions $\ds f(x)=P_\e \eta_{0j}(x)-\eta_{0j}(x)-\sum_{i=1}^m\gamma_{ij}G(x,\xi_i)$ and $\ds \ti f(x)=P_\e \eta_{j}(x)-\eta_{j}(x)-{8\pi\over 3}\al_jH(x,\xi_j)+\sum_{i=1}^m\ti\gamma_{ij}G(x,\xi_i)$, we have that the following fact, as shown in \cite{EFP}.

\begin{lemma}
There hold
$$P_\e \eta_{0j}=\eta_{0j} + \sum_{i=1}^m \gamma_{ij} G(x,\xi_i)+O(\rho^{\ti\sigma})\quad\text{ and }\quad P_\e \eta_{j}=\eta_{j}+{8\pi\over 3}\al_j H(x,\xi_j)-\sum_{i=1}^m \ti\gamma_{ij} G(x,\xi_i)+O(\rho^{\ti\sigma})$$ 
uniformly in $\Om_\e$ for some $\ti\sigma>0$.
\end{lemma}

\begin{claim}\label{claim3}
There hold that $a_{j}=0$ for all $j=1,\dots,m$.
\end{claim}
\begin{proof}[\dem] To this aim let us construct suitable tests functions and from the assumption on $h_n$, $|\log \rho_n|\ \|h_n\|_*=o(1)$, we get the additional relation
\begin{equation}\label{aj0}
a_j\int_{\R^2} {2\al_j^2|y|^{\al_j-2}\over (1+|y|^{\al_j})^2}Y_{0j}^2-{\al_j(\al_j-2)\over 3} \	 a_j\int_{\R^2 }  {2\al_j^2|y|^{\al_j-2}\over (1+|y|^{\al_j})^2 }Y_{0j} \log |y| \, dy=0,
\end{equation}
which implies $a_j=0$ as claimed, since
$$\int_{\R^2} {2\al_j^2|y|^{\al_j-2}\over (1+|y|^{\al_j})^2}Y_{0j}^2=\int_{\R^2} {2\al_j^2|y|^{\al_j-2}\over (1+|y|^{\al_j})^2}\({1-|y|^{\al_j} \over 1+|y|^{\al_j} }\)^2dy={4\pi\over 3}\al_j$$
and 
$$\int_{\R^2} {2\al_j^2|y|^{\al_j-2}\over (1+|y|^{\al_j})^2}Y_{0j}\log|y|=\int_{\R^2} {2\al_j^2|y|^{\al_j-2}\over (1+|y|^{\al_j})^2}\ {1-|y|^{\al_j} \over 1+|y|^{\al_j} }\ \log|y|\, dy= -4\pi .$$

Next, as in \cite[Claim 3, Section 4]{EFP}, we define the following test function $P_\e Z_j$, where $Z_j=\eta_j+\gamma_j^*\eta_{0j}$ and $\gamma_j^*$ is given by
$$\gamma_j^*=\dfrac{\ds \dfrac{\ti\gamma_{jj}}{2\pi}\log\de_j +\(\dfrac{8\pi}{3}\al_j-\ti\gamma_{jj}\) H(\xi_j,\xi_j) -  \sum_{i=1,i\ne j} ^m \ti\gamma_{ij} G(\xi_i,\xi_j) }{\ds 1-\gamma_{jj}H(\xi_j,\xi_j)- \sum_{i=1,i\ne j} ^m \gamma_{ij} G(\xi_i,\xi_j) + \dfrac{\gamma_{jj}}{2\pi}\log\de_j},$$
so that,
\begin{equation}\label{gamajs}
\gamma_j^*= \(\dfrac{8\pi}{3}\al_j-\ti\gamma_{jj}+ \gamma_{jj}\gamma_j^*\) H(\xi_j,\xi_j) -  \sum_{i=1,i\ne j} ^m (\ti\gamma_{ij} - \gamma_j^* \gamma_{ij}) G(\xi_i,\xi_j) +\dfrac{1}{2\pi}(\ti\gamma_{jj}- \gamma_{j}^* \gamma_{jj})\log\de_j .
\end{equation}
From the \eqref{choice1} and the expansions for $\gamma_{jj}$ and $\ti\gamma_{jj}$ we obtain that
$$\gamma_j^*=\dfrac{\ds \dfrac{\ti\gamma_{jj}}{2\pi}\log\de_j+O(1)}{\ds1+\dfrac{ \gamma_{jj}}{2\pi}\log\de_j + O\Big({1\over |\log\rho|} \Big)} = - {\al_j-2\over 3}\log \rho +O(1).$$
Notice that $P_\e Z_j$ expands as
\begin{align}\label{pzj}
P_\e Z_j
=&\,Z_{j}+{8\pi\over 3}\al_j H(x,\xi_j)- \ti\gamma_{jj}\(-{1\over 2\pi}\log|x - \xi_i| + H(x,\xi_j)\) -\sum_{i=1,i\ne j}^m \ti\gamma_{ij} G(x,\xi_i) \nonumber\\
&+O(\rho^{\ti\sigma}) + \gamma_j^*\lf[\gamma_{jj}\(-{1\over 2\pi}\log|x - \xi_i| + H(x,\xi_j)\) + \sum_{i=1,i\ne j}^m \gamma_{ij} G(x,\xi_i)+O(\rho^{\ti\sigma})\rg] \nonumber\\
=&\,Z_{j}+\( {8\pi\over 3} \al_j - \ti\gamma_{jj} + \gamma_{jj}\gamma_j^*\) H(x,\xi_j) -\sum_{i=1,i\ne j}^m \(\ti\gamma_{ij} -  \gamma_{ij}\gamma_j^* \)G(x,\xi_i)\\
& + {1\over 2\pi}\(\ti\gamma_{jj}-\gamma_j^*\gamma_{jj}\) \log|x - \xi_j| +O(\rho^{\ti\sigma}) + \gamma_{j}^* O(\rho^{\ti\sigma})\nonumber
\end{align}
Hence, multiplying equation \eqref{eqphin} by $P_\e Z_j$ and integrating by parts we obtain that
\begin{equation*}
\begin{split}
\int_{\Om_\e} hPZ_{j}=&\ \int_{\Om_\e}\lap Z_{j}\phi + \int_{\Om_\e} W\phi  P_\e Z_{j}  ,
\end{split}
\end{equation*}
in view of $P_\e Z_{j}=0$ on $\fr\Om_{\e_n}$ and $\ds \int_{\Om_\e}\lap\phi P_\e Z_{j}=\int_{\Om_\e}\phi\; \lap P_\e Z_{j}$. Furthermore, we have that
\begin{equation*}
\begin{split}
\int_{\Om_\e} hPZ_{j} 
=&\,\int_{\Om_\e} \phi\lf[|x-\xi_j|^{\al_j-2}e^{w_j}Z_{0j} - |x-\xi_j|^{\al_j-2} e^{w_j}\eta_j+\gamma_j^*\( - |x-\xi_j|^{\al_j-2}e^{w_j}  - |x-\xi_j|^{\al_j-2} e^{w_j}\eta_{0j}\) \rg]\\
& +\int_{\Om_\e} W\phi  P_\e Z_{j}\\
=&\,\int_{\Om_\e} \phi |x-\xi_j|^{\al_j-2}e^{w_j}Z_{0j} +\int_{\Om_\e} |x-\xi_j|^{\al_j-2} e^{w_j} \( PZ_j-Z_j - \gamma_j^*\) \\
& +\int_{\Om_\e} \(W - |x-\xi_j|^{\al_j-2}e^{w_j}\)  P_\e Z_{j} \phi  
\end{split}
\end{equation*}
Now, estimating every integral term we find that
$$\int_{\Om_\e} hPZ_{j}=O\lf(\, |\log\rho|\,\|h\|_p\rg)=o\lf( 1 \rg),$$
in view of $PZ_{j}=O(|\log\rho|)$ and $G(x,\xi_k)=O(|\log\e_k|)$. Next, scaling we obtain that
$$\int_{\Om_\e} \phi |x-\xi_j|^{\al_j-2}e^{w_j}Z_{0j}=\int_{\Om_{j,n } }  {2\al_j^2|y|^{\al_j-2}\over (1+|y|^{\al_j})^2 }\Phi_{j,n}Y_{0j}\, dy =a_j\int_{\R^2 }  {2\al_j^2|y|^{\al_j-2}\over (1+|y|^{\al_j})^2 }Y_{0j}^2 \, dy +o(1)$$
Also, by using \eqref{gamajs} and \eqref{pzj} we get that
\begin{equation*}
\begin{split}
 \int_{\Om_\e} |x-\xi_j|^{\al_j-2} e^{w_j}  \( PZ_j-Z_j - \gamma_j^*\) = &\,
\int_{\Om_\e} |x-\xi_j|^{\al_j-2} e^{w_j} O(\e^{\ti\sigma}) +\int_{\Om_{j,n } }  {2\al_j^2|y|^{\al_j-2}\over (1+|y|^{\al_j})^2 }\Phi_{j,n} O(\de_j|y|)\, dy \\
&\, + {1\over 2\pi} \(\ti\gamma_{jj}-\gamma_j^*\gamma_{jj}\)\int_{\Om_{j,n } }  {2\al_j^2|y|^{\al_j-2}\over (1+|y|^{\al_j})^2 }\Phi_{j,n} \log |y| \, dy\\
=&\,-{\al_j(\al_j-2)\over 3} \	 a_j\int_{\R^2 }  {2\al_j^2|y|^{\al_j-2}\over (1+|y|^{\al_j})^2 }Y_{0j} \log |y| \, dy + o(1),
\end{split}
\end{equation*}
in view of $\ds {1\over 2\pi} \(\ti\gamma_{jj}-\gamma_j^*\gamma_{jj}\)=-{\al_j(\al_j-2)\over 3} + O\Big({1\over |\log\rho|} \Big)$. Furthermore, we have that
\begin{equation*}
\begin{split}
\int_{\Om_\e} \(W - |x-\xi_j|^{\al_j-2}e^{w_j}\)  P_\e Z_{j} \phi&=\int_{B_j-\xi_j\over\de_j } {2\al_j^2|y|^{\al_j-2}\over (1+|y|^{\al_j})^2}O(\de_j|y|+ \rho^{\ti\sigma} ) \Phi_{j,n} P_\e Z_j(\xi_j+\de_j y)\, dy\\
&\quad +\sum_{i=1, i\ne j}^m \int_{B_i-\xi_i\over\de_i } {2\al_i^2|y|^{\al_i-2}\over (1+|y|^{\al_i})^2}[1+O(\de_i|y|+ \rho^{\ti\sigma} )] \Phi_{i,n} P_\e Z_j(\xi_i+\de_i y)\, dy \\
&\quad +O(\rho^{\ti\sigma} |\log\rho|)\\
=&\, O\(\rho^{\ti \sigma} |\log\rho| \int_{B_i-\xi_i\over\de_i} {2\al_j^2|y|^{\al_j-2}\over (1+|y|^{\al_j})^2}( |y|+ 1 )|\Phi_{i,n}|\,dy \)+ o(1)\\
=&\, o(1),
\end{split}
\end{equation*}
since from \eqref{pzj}, $P_\e Z_j(\xi_j+\de_j y)=O(|\log\rho|),$ and for $i\ne j$ and $y\in \de_i^{-1}(B_i-\xi_i)$ it holds
\begin{equation*}
\begin{split}
P_\e Z_j(\xi_i+\de_i y) 
=&\, {8\pi\over 3}\al_j G(\xi_i,\xi_j ) +O(\de_j^{\al_j} + \de_i|y|) - \(\ti\gamma_{ij} -  \gamma_{ij}\gamma_j^* \)\(-{1\over 2\pi} \log|\de_iy| + H(\xi_i+\de_i y,\xi_i)\) \\
&-\sum_{k=1,k\ne i}^m \(\ti\gamma_{kj} -  \gamma_{kj}\gamma_j^* \)\( G(\xi_i,\xi_k) + O(\de_i |y|)\) \\ 
=&\, {1\over 2\pi}\(\ti\gamma_{ij} -  \gamma_{ij}\gamma_j^* \)  \log|\de_iy|  + {8\pi\over 3}\al_j G(\xi_i,\xi_j ) - \(\ti\gamma_{ij} -  \gamma_{ij}\gamma_j^* \) H(\xi_i,\xi_i) \\
&- \sum_{k=1,k\ne i}^m \(\ti\gamma_{kj} -  \gamma_{kj}\gamma_j^* \)G(\xi_i,\xi_k) + O(\de_i |y|)
\end{split}
\end{equation*}
and
\begin{equation*}
\begin{split}
\int_{B_i-\xi_i\over\de_i } &{2\al_i^2|y|^{\al_i-2}\over (1+|y|^{\al_i})^2}[1+O(\de_i|y|+ \rho^{\ti\sigma} )] \Phi_{i,n} P_\e Z_j(\xi_i+\de_i y)\, dy\\
=&\,{1\over 2\pi}  \(\ti\gamma_{ij} -  \gamma_{ij}\gamma_j^* \)\log\de_i  \int_{B_i-\xi_i\over\de_i } {2\al_i^2|y|^{\al_i-2}\over (1+|y|^{\al_i})^2} \Phi_{i,n}  \, dy \\
&\,+ {1\over 2\pi}  \(\ti\gamma_{ij} -  \gamma_{ij}\gamma_j^* \)  \int_{B_i-\xi_i\over\de_i } {2\al_i^2|y|^{\al_i-2}\over (1+|y|^{\al_i})^2}\Phi_{i,n}   \log| y| \, dy \\
&\, - \(\text{bounded constant} \) \int_{B_i-\xi_i\over\de_i } {2\al_i^2|y|^{\al_i-2}\over (1+|y|^{\al_i})^2}[1+O(\de_i|y|+ \rho^{\ti\sigma} )] \Phi_{i,n}\, dy \\
&\, + O\( \de_i \int_{B_i-\xi_i\over\de_i } {2\al_i^2|y|^{\al_i-2}\over (1+|y|^{\al_i})^2}[|y|+\de_i|y|^2]\, dy\) =o(1)\\
\end{split}
\end{equation*}
Therefore, we conclude \eqref{aj0} and hence, $a_j=0$ for all $j=1,\dots,m$.
\end{proof}

Now, by using Claims \ref{claim2} and  \ref{claim3} we deduce that $\Psi_{j,n}$ converges to zero weakly in $H_{\al_j}(\R^2)$ and strongly in $L_{\al_j}(\R^2)$ as $n\to+\infty$. Thus, we arrives at a contradiction with \eqref{npsi} and it follows the priori estimate $\|\phi\|\le C|\log \rho | \, \|h\|_p$. It only remains to prove the solvability assertion. To this purpose consider the space $H=H_0^1(\Om_\e)$ endowed with
the usual inner product $\ds [\phi,\psi]=\int_{\Om_\e} \grad\phi\grad\psi.$
Problem \eqref{pl} expressed in weak form is equivalent to that
of finding a $\phi\in H$ such that
$$[\phi,\psi]=\int_{\Om_\e} \lf[ W\phi-h\rg]\psi,\qquad\text{for all }\psi\in H.$$ With the aid of Riesz's representation theorem, this equation gets rewritten in
$H$ in the operator form $\phi=\ml{K}(\phi)+\ti h$, for certain
$\ti h\in H$, where $\ml{K}$ is a compact operator in $H$.
Fredholm's alternative guarantees unique solvability of this
problem for any $h$ provided that the homogeneous equation
$\phi=\ml{K}(\phi)$ has only the zero solution in $H$. This last
equation is equivalent to \eqref{pl} with $h\equiv 0$. Thus,
existence of a unique solution follows from the a priori estimate
\eqref{estphi}. This finishes the proof.
\end{proof}

\bigskip

\begin{center}
\textsc{Acknowledgements}
\end{center}
\noindent The author would like to thank Prof. Pistoia (U. Roma ``La Sapienza", Italia) and Prof. Esposito (U. Roma Tre, Italia) for many interesting discussions about these type of problems. The author has been supported by grant Fondecyt Regular Nº 1201884, Chile.

\end{document}